\documentclass[reqno,10pt,centertags]{amsart}
\usepackage{hyperref}
\usepackage{amsmath,amsthm,amscd,amssymb,latexsym,esint,upref,stmaryrd,
enumerate,color,verbatim,yfonts,mathrsfs} 

\newcommand*{\mailto}[1]{\href{mailto:#1}{\nolinkurl{#1}}}

\newcommand{\bbC}{{\mathbb{C}}}

\newcommand{\bbN}{{\mathbb{N}}}

\newcommand{\bbR}{{\mathbb{R}}}

\newcommand{\cB}{{\mathcal B}}

\newcommand{\cD}{{\mathcal D}}

\newcommand{\cF}{{\mathcal F}}

\newcommand{\cH}{{\mathcal H}}

\newcommand{\cW}{{\mathcal W}}
\newcommand{\cX}{{\mathcal X}}

\newcommand{\gq}{{\mathfrak{q}}}

\DeclareMathOperator{\supp}{supp}

\DeclareMathOperator{\ran}{ran}
\DeclareMathOperator{\dom}{dom}

\DeclareMathOperator*{\slim}{s-lim}

\DeclareMathOperator*{\sgn}{sgn}

\newcommand{\loc}{\text{\rm{loc}}}

\newcommand{\no}{\notag}
\newcommand{\lb}{\label}
\newcommand{\f}{\frac}

\newcommand{\ol}{\overline}

\newcommand{\wti}{\widetilde}

\newcommand{\hatt}{\widehat} 
\newcommand{\bi}{\bibitem}
\newcommand{\dott}{\,\cdot\,}

\let\geq\geqslant
\let\leq\leqslant

\makeatletter
\def\theequation{\@arabic\c@equation}

\allowdisplaybreaks 
\numberwithin{equation}{section}

\newtheorem{theorem}{Theorem}[section]

\newtheorem{lemma}[theorem]{Lemma}
\newtheorem{corollary}[theorem]{Corollary}

\newtheorem{hypothesis}[theorem]{Hypothesis}

\theoremstyle{remark}

\begin{document}

\numberwithin{equation}{section}
\allowdisplaybreaks

\title[Higher-Order Krein Laplacians]{A Bound for the Eigenvalue Counting Function for Higher-Order Krein Laplacians on Open Sets} 
  
\author[F.\ Gesztesy]{Fritz Gesztesy}  
\address{Department of Mathematics,
University of Missouri, Columbia, MO 65211, USA}
\email{\mailto{gesztesyf@missouri.edu}}
\urladdr{\url{http://www.math.missouri.edu/personnel/faculty/gesztesyf.html}}

\author[A.\ Laptev]{Ari Laptev}  
\address{Department of Mathematics, Imperial College London, Huxley Building, 180
Queen�s Gate, London SW7 2AZ, UK}  
\email{\mailto{a.laptev@imperial.ac.uk}}
\urladdr{\url{http://www2.imperial.ac.uk/~alaptev/}}

\author[M.\ Mitrea]{Marius Mitrea}
\address{Department of Mathematics,
University of Missouri, Columbia, MO 65211, USA}
\email{\mailto{mitream@missouri.edu}}
\urladdr{\url{http://www.math.missouri.edu/personnel/faculty/mitream.html}}

\author[S.\ Sukhtaiev]{Selim Sukhtaiev}
\address{Department of Mathematics, University of
Missouri, Columbia, MO 65211, USA}
\email{\mailto{sswfd@mail.missouri.edu}}


\thanks{Work of M.\,M.\ was partially supported by the Simons Foundation Grant $\#$\,281566
and by a University of Missouri Research Leave grant.} 
\thanks{{\it Mathematical Results in Quantum Mechanics, QMath12 Proceedings}, P.\ Exner, W.\ K\"onig, and H.\ Neidhardt (eds), World Scientific, Singapore, to appear.}

\date{\today}

\begin{abstract} 
For an arbitrary nonempty, open set $\Omega \subset \bbR^n$, $n \in \bbN$, 
of finite (Euclidean) volume, we consider 
the minimally defined higher-order Laplacian $(- \Delta)^m\big|_{C_0^{\infty}(\Omega)}$, 
$m \in \bbN$, and its Krein--von Neumann extension $A_{K,\Omega,m}$ in $L^2(\Omega)$. With $N(\lambda ,A_{K,\Omega,m})$, $\lambda > 0$, denoting the eigenvalue counting function corresponding to the strictly positive eigenvalues of $A_{K,\Omega,m}$, we derive the bound   
$$
N(\lambda ,A_{K,\Omega,m}) \leq (2 \pi)^{-n} v_n |\Omega| 
\{1 + [2m/(2m+n)]\}^{n/(2m)} \lambda^{n/(2m)}, \quad \lambda > 0, 
$$
where $v_n := \pi^{n/2}/\Gamma((n+2)/2)$ denotes the (Euclidean) volume of the unit 
ball in $\bbR^n$. 

The proof relies on variational considerations and exploits the fundamental link between 
the Krein--von Neumann extension and an underlying (abstract) buckling problem.
\end{abstract}

\maketitle 


\section{Introduction}  \lb{s1}

To set the stage, suppose that $S$ is a densely defined, symmetric, closed operator with nonzero deficiency indices in a separable complex Hilbert space $\cH$ that satisfies 
\begin{equation}
S\geq \varepsilon I_{\cH} \, \text{ for some } \, \varepsilon >0.     \lb{1.1}
\end{equation}
Then, according to M.\ Krein's celebrated 1947 paper \cite{Kr47}, among all nonnegative self-adjoint extensions of $S$, there exist two distinguished ones, $S_F$, the Friedrichs extension of $S$ and  
$S_K$, the Krein--von Neumann extension of $S$, which are, respectively, the largest and 
smallest such extension (in the sense of quadratic forms). In particular, a nonnegative self-adjoint 
operator $\widetilde{S}$ is a self-adjoint extension of $S$ if and only if $\widetilde{S}$ 
satisfies 
\begin{equation} 
S_K\leq\widetilde{S}\leq S_F  
\end{equation}
(again, in the sense of quadratic forms).

An abstract version of \cite[Proposition\ 1]{Gr83}, presented in \cite{AGMST10}, describing the following intimate connection between the nonzero eigenvalues of $S_K$, and a suitable abstract buckling problem, can be summarized as  follows:
\begin{align}
& \text{There exists $0 \neq v_{\lambda} \in \dom(S_K)$ satisfying } \,  
S_K v_{\lambda} = \lambda v_{\lambda}, \quad \lambda \neq 0,   \lb{1.1a} \\
& \text{if and only if }   \no \\ 
& \text{there exists a $0 \neq u_{\lambda} \in \dom(S^* S)$ such that } \, 
S^* S u_{\lambda} = \lambda S u_{\lambda},   \lb{1.1b} 
\end{align}
and the solutions $v_{\lambda}$ of \eqref{1.1a} are in one-to-one correspondence with the 
solutions $u_{\lambda}$ of \eqref{1.1b} given by the pair of formulas
\begin{equation}
u_{\lambda} = (S_F)^{-1} S_K v_{\lambda},    \quad  
v_{\lambda} = \lambda^{-1} S u_{\lambda}.   \lb{1.1c}
\end{equation}
As briefly recalled in Section \ref{s2}, \eqref{1.1b} represents an abstract buckling problem. 
The latter has been the key in all attempts to date in proving Weyl-type asymptotics for 
eigenvalues of $S_K$ when $S$ represents an elliptic partial differential operator in 
$L^2(\Omega)$. In fact, it is convenient to go one step further and replace the abstract buckling eigenvalue problem \eqref{1.1b} by the variational formulation,  
\begin{align}
\begin{split} 
& \text{there exists $u_{\lambda}\in\dom(S)\backslash\{0\}$ such that}   \\
& \quad \text{${\mathfrak{a}}(w,u_{\lambda})=\lambda\,{\mathfrak{b}}(w,u_{\lambda})$ 
for all $w\in\dom(S)$},  
\end{split} 
\end{align}
where the symmetric forms $\mathfrak{a}$ and $\mathfrak{b}$ in $\cH$ are defined by 
\begin{align}   
\mathfrak{a}(f,g) & :=(Sf,Sg)_{\cH},\quad f,g\in\dom(\mathfrak{a}):=\dom(S),    \\ 
\mathfrak{b}(f,g) & :=(f,Sg)_{\cH},\quad f,g\in\dom(\mathfrak{b}):=\dom(S).    
\end{align}

In the present context of higher-order Krein Laplacians, the role of $S$ will be played by the 
closure of the minimally defined operator in $L^2(\Omega)$,
\begin{equation}
A_{min,\Omega,m} := (- \Delta)^m, \quad \dom (A_{min,\Omega,m}) := C_0^{\infty}(\Omega),      
\end{equation}
under the assumption that $\emptyset \neq \Omega \subset \bbR^n$ has finite 
(Euclidean) volume ($|\Omega| < \infty$). This closure, 
$\ol{A_{min,\Omega,m}}$, is denoted by $A_{\Omega,m}$ and explicitly given by
\begin{equation}
A_{\Omega,m} = (- \Delta)^m, \quad \dom (A_{\Omega,m}) = \mathring W^{2m} (\Omega).   
\end{equation} 
The Krein--von Neumann and Friedrichs extensions of $A_{\Omega,m}$ will then be denoted by 
$A_{K, \Omega, m}$ and $A_{F, \Omega, m}$, respectively. (To provide a quick example, we 
note that in the special case $m=n=1$ and $\Omega = (a,b)$, $-\infty <a < b < \infty$, the 
boundary condition associated with $A_{K, (a,b), 1}$ explicitly reads 
$v'(a)=v'(b)=[v(b)-v(a)]/(b-a)$, and that for $A_{F, (a,b), 1}$ is of course the Dirichlet boundary 
condition $v(a) = v(b) = 0$.)

Since $A_{K,\Omega,m}$ has purely discrete spectrum in $(0,\infty)$ bounded away from zero 
by $\varepsilon > 0$ (cf.\ Theorem \ref{t2.4}), let $\{\lambda_{K, \Omega, j}\}_{j\in\bbN}\subset(0,\infty)$ be 
the strictly positive eigenvalues of $A_{K,\Omega,m}$ enumerated in nondecreasing order, counting multiplicity, and let
\begin{equation} 
N(\lambda,A_{K,\Omega,m}):=\#\{j\in\bbN\,|\,0<\lambda_{K,\Omega,j} < \lambda\}, 
\quad \lambda > 0, 
\end{equation}
be the eigenvalue distribution function for $A_{K,\Omega,m}$ (which takes into account only strictly positive eigenvalues of $A_{K,\Omega,m}$). The function 
$N(\, \cdot \, ,A_{K,\Omega,m})$ is the principal object of this note. 
Similarly, $N(\lambda,A_{F,\Omega,m})$, $\lambda >0$, denotes the eigenvalue counting 
function for $A_{F,\Omega,m}$. 

In Section \ref{s2} we recall the basic abstract facts on the Friedrichs extension, $S_F$ 
and the Krein--von Neumann extension $S_K$ of a strictly positive, closed, symmetric operator 
$S$ in a complex, separable Hilbert space $\cH$ and describe the intimate link between the 
Krein--von Neumann extension and an underlying abstract buckling problem. Section \ref{s3} then focuses on the concrete case of higher-order Laplacians $(- \Delta)^m$, $m \in \bbN$, on 
open, finite (Euclidean) volume subsets $\Omega \subset \bbR^n$ (without 
 imposing any constraints on $\Omega$ in the case where $\Omega$ is bounded) 
 and derives the bound 
\begin{equation}
N(\lambda ,A_{K,\Omega,m}) \leq (2 \pi)^{-n} v_n |\Omega| 
\{1 + [2m/(2m+n)]\}^{n/(2m)} \lambda^{n/(2m)},  \quad 
 \lambda > 0,     \lb{1.12} 
\end{equation}
where $v_n := \pi^{n/2}/\Gamma((n+2)/2)$ denotes the (Euclidean) volume of the unit 
ball in $\bbR^n$. 
We remark that the power law behavior $\lambda^{n/(2m)}$ coincides with the one in the known 
Weyl asymptotic behavior. This in itself is perhaps not surprising as it is {\it a priori} known that 
\begin{equation}
N(\lambda ,A_{K,\Omega,m}) \leq N(\lambda ,A_{F,\Omega,m}), \quad \lambda > 0,  \lb{1.13}
\end{equation} 
and $N(\lambda ,A_{F,\Omega,m})$ is known to have the power law behavior 
$\lambda^{n/(2m)}$ (cf.\ \eqref{4.3}, due to \cite{La97}, which in turn extends the corresponding 
result in \cite{LY83} in the case $m=1$). We emphasize that \eqref{1.13} is not in conflict with 
variational eigenvalue estimates since    
$N(\lambda ,A_{K,\Omega,m})$ only counts the strictly positive eigenvalues of 
$A_{K,\Omega,m}$ less than $\lambda > 0$ and hence avoids taking into account the (generally,  
infinite-dimensional) null space of $A_{K,\Omega,m}$. 
Rather than using known estimates for $N(\, \cdot \, ,A_{F,\Omega,m})$ 
(cf., e.g., \cite{BS70}, \cite{BS71}, \cite{BS72}, \cite{BS73}, \cite{BS79}, \cite{BS80}, 
\cite{Ge13}, \cite{GLW11}, \cite{HH08}, \cite{HH11}, \cite{La97}, \cite{Li80}, \cite{LY83}, 
\cite{Me77}, \cite{NS05}, \cite{Ro71}, \cite{Ro72}, \cite{Ro76}, \cite{Sa01}, \cite{We08}), 
we will use the one-to-one correspondence of nonzero eigenvalues 
of $A_{K,\Omega,m}$ with the eigenvalues of its underlying buckling problem 
(cf.\ \eqref{1.1a}--\eqref{1.1c}) and estimate the eigenvalue counting function for the latter 
in Section \ref{s3}. In our final Section \ref{s4} we briefly discuss the superiority of the buckling problem based bound \eqref{1.12} over the known estimates for $N(\, \cdot \, ,A_{F,\Omega,m})$.

Since Weyl asymptotics for $N(\, \cdot \, ,A_{K,\Omega,m})$ and $N(\, \cdot \,  ,A_{F,\Omega,m})$  
are not considered in this paper we just refer to the monographs \cite{Le90} and \cite{SV97}, but 
note that very detailed bibliographies on this subject appeared in \cite{AGMT10} and 
\cite{AGMST13}. At any rate, the best known result on Weyl asymptotics for 
$N(\, \cdot \, ,A_{K,\Omega,m})$ to date is proven for bounded Lipschitz domains 
\cite{BGMM14}, whereas the estimate \eqref{1.12} assumes no regularity of $\Omega$ at all. 

We conclude this introduction by summarizing the notation used in this paper. Throughout this 
paper, the symbol $\cH$ is reserved to denote a separable complex Hilbert space with  
$(\dott,\dott)_{\cH}$ the scalar product in $\cH$ (linear in the second argument), and $I_{\cH}$ 
the identity operator in $\cH$. Next, let $T$ be a linear operator mapping (a subspace of) a
Banach space into another, with $\dom(T)$ and $\ran(T)$ denoting the domain and range of $T$. 
The closure of a closable operator $S$ is denoted by $\ol S$. The kernel (null space) of $T$ is 
denoted by $\ker(T)$. The spectrum, point spectrum (i.e., the set of eigenvalues), discrete 
spectrum, essential spectrum, and resolvent set of a closed linear operator in $\cH$ will be 
denoted by $\sigma(\cdot)$, $\sigma_{p}(\cdot)$, $\sigma_{d}(\cdot)$, $\sigma_{ess}(\cdot)$, 
and $\rho(\cdot)$, respectively. The symbol $\slim$ abbreviates the limit in the strong 
(i.e., pointwise) operator topology (we also use this symbol to describe strong limits in $\cH$). 

The Banach spaces of bounded and compact linear operators on $\cH$ are
denoted by $\cB(\cH)$ and $\cB_\infty(\cH)$, respectively. Similarly,
the Schatten--von Neumann (trace) ideals will subsequently be denoted
by $\cB_p(\cH)$, $p\in (0,\infty)$. 
In addition, $U_1\dotplus U_2$ denotes the direct sum of the subspaces $U_1$ 
and $U_2$ of a Banach space $\cX$. Moreover, $\cX_1\hookrightarrow\cX_2$ denotes the continuous embedding of the Banach space $\cX_1$ into the Banach space $\cX_2$. 

The symbol $L^2(\Omega)$, with $\Omega\subseteq\bbR^n$ open, $n\in\bbN$, 
is a shortcut for $L^2(\Omega,d^n x)$, whenever the $n$-dimensional Lebesgue measure is 
understood. For brevity, the identity operator in $L^2(\Omega)$ will typically be denoted 
by $I_{\Omega}$. The symbol $\cD(\Omega)$ is reserved for the set 
of test functions $C_0^{\infty}(\Omega)$ on $\Omega$, equipped with the standard 
inductive limit topology, and $\cD'(\Omega)$ represents its dual space, the set of distributions 
in $\Omega$. The cardinality of a set $M$ is abbreviated by $\#(M)$. In addition, 
we define $\bbN_0:=\bbN\cup\{0\}$, so that $\bbN_0^n$ becomes the collection of all multi-indices 
with $n$ components. As is customary, for each 
$\alpha=(\alpha_1,...,\alpha_n)\in{\mathbb{N}}_0^n$ we denote by 
$|\alpha|:=\alpha_1+\cdots+\alpha_n$ the length of $\alpha$, and set 
$\alpha!:=\alpha_1!\cdots\alpha_n!$.

Moreover, $A\approx B$ signifies the existence of a finite constant $C\geq 1$,
independent of the main parameters entering the quantities $A,B$, such that 
$C^{-1} A \leq B \leq CA$.
  
Finally, a notational comment: For obvious reasons, which have their roots in quantum 
mechanical applications, we will, with a slight abuse of notation, dub the expression 
$-\Delta=-\sum_{j=1}^n\partial_j^2$ (rather than $\Delta$) as the ``Laplacian'' in this paper.

\section{Basic Facts on the Krein--von Neumann 
extension and the Associated Abstract Buckling Problem} 
\lb{s2}

In this preparatory section we recall the basic facts on the Krein--von Neumann extension of 
a strictly positive operator $S$ in a complex, separable Hilbert space $\cH$ and its associated 
abstract buckling problem as discussed in \cite{AGMT10, AGMST10}. 
For an extensive survey of this circle of ideas and an exhaustive list of references 
as well as pertinent historical comments we refer to \cite{AGMST13}.

To set the stage, we denote by $S$ a linear, densely defined, symmetric 
(i.e., $S\subseteq S^*$), and closed operator in $\cH$ throughout this section. 
We recall that $S$ is called {\it nonnegative} provided $(f,Sf)_\cH\geq 0$ for all $f\in\dom(S)$.
The operator $S$ is called {\it strictly positive}, if for some $\varepsilon>0$ one has
$(f,Sf)_\cH\geq\varepsilon\|f\|_{\cH}^2$ for all $f\in\dom(S)$; 
one then writes $S\geq\varepsilon I_{\cH}$.  
Next, we recall that two nonnegative, self-adjoint operators $A,B$ in $\cH$ 
satisfy $A\leq B$ (in the sense of forms) if 
\begin{align}\lb{AleqBjussi1}
\dom\big(B^{1/2}\big)\subset\dom\big(A^{1/2}\big)        
\end{align}
and
\begin{align}\lb{AleqBjussi2}
\big(A^{1/2}f,A^{1/2}f\big)_{\cH}\leq\big(B^{1/2}f,B^{1/2}f\big)_{\cH},
\quad f\in\dom\big(B^{1/2}\big). 
\end{align}
We also recall (\cite[Section\ I.6]{Fa75}, \cite[Theorem\ VI.2.21]{Ka80}) that for 
$A$ and $B$ both self-adjoint and nonnegative in $\cH$ one has 
\begin{equation}
0\leq A\leq B\,\text{ if and only if }\,
(B+a I_\cH)^{-1}\leq(A+a I_\cH)^{-1}\,\text{ for all }\,a>0.      
\end{equation}
Moreover, we note the useful fact that $\ker(A)=\ker(A^{1/2})$. 

The following is a fundamental result to be found in M.\ Krein's celebrated 1947 paper
\cite{Kr47} (cf.\ also Theorems~2 and 5--7 in the English summary on page 492): 
 
\begin{theorem}\lb{t2.1}
Assume that $S$ is a densely defined, closed, nonnegative operator in $\cH$. Then, among all 
nonnegative self-adjoint extensions of $S$, there exist two distinguished ones, $S_K$ and $S_F$, 
which are, respectively, the smallest and largest such extension {\rm (}in the sense of 
\eqref{AleqBjussi1}--\eqref{AleqBjussi2}{\rm )}. Furthermore, a nonnegative self-adjoint 
operator $\widetilde{S}$ is a self-adjoint extension of $S$ if and only if $\widetilde{S}$ 
satisfies 
\begin{equation}\lb{Fr-Sa}
S_K\leq\widetilde{S}\leq S_F.
\end{equation}
In particular, \eqref{Fr-Sa} determines $S_K$ and $S_F$ uniquely. 
In addition,  if $S\geq\varepsilon I_{\cH}$ for some $\varepsilon>0$, one has 
$S_F\geq\varepsilon I_{\cH}$, and 
\begin{align}\lb{SF}
\dom(S_F) & =\dom(S)\dotplus (S_F)^{-1}\ker (S^*),   \\ 
\dom(S_K) & =\dom(S)\dotplus\ker(S^*),    
\lb{SK}   \\
\dom(S^*) & =\dom(S)\dotplus (S_F)^{-1}\ker(S^*)\dotplus\ker(S^*)    \no \\
& =\dom(S_F)\dotplus\ker(S^*),    
\lb{S*} 
\end{align}
and 
\begin{equation}\lb{Fr-4Tf}
\ker(S_K)=\ker\big((S_K)^{1/2}\big)=\ker(S^*)=\ran(S)^{\bot}.
\end{equation} 
\end{theorem}

One calls $S_K$ the {\it Krein--von Neumann extension} of $S$ and $S_F$ the 
{\it Friedrichs extension} of $S$. We also recall that 
\begin{equation}\label{yaf74df}
S_F=S^*|_{\dom(S^*)\cap\dom((S_{F})^{1/2})}.
\end{equation}
Furthermore, if $S\geq\varepsilon I_{\cH}$ then \eqref{SK} implies 
\begin{equation}\label{yaere}
\ker(S_K)=\ker\big((S_K)^{1/2}\big)=\ker(S^*)=\ran(S)^{\bot}.
\end{equation} 

For abstract results regarding the parametrization of all nonnegative self-adjoint extensions 
of a given strictly positive, densely defined, symmetric operator we refer the reader to 
Krein \cite{Kr47}, Vi{\v s}ik \cite{Vi63}, Birman \cite{Bi56}, Grubb \cite{Gr68,Gr70}, 
subsequent expositions due to Alonso and Simon \cite{AS80}, Faris \cite[Sect.\ 15]{Fa75}, 
and \cite[Sect.~13.2]{Gr09}, \cite[Ch.~13]{Sc12}, and Derkach and Malamud \cite{DM91}, 
Malamud \cite{Ma92}, see also \cite[Theorem~9.2]{GM11}.

Let us collect a basic assumption which will be imposed in the rest of this section.

\begin{hypothesis}\lb{h2.2}
Suppose that $S$ is a densely defined, symmetric, closed operator with nonzero deficiency 
indices in $\cH$ that satisfies $S\geq\varepsilon I_{\cH}$ for some $\varepsilon >0$. 
\end{hypothesis}

For subsequent purposes we note that under Hypothesis~\ref{h2.2}, one has 
\begin{equation}\lb{jajagutgut}
\dim\big(\ker(S^*-z I_{\cH})\big)=\dim\big(\ker(S^*)\big),
\quad z\in\bbC\backslash[\varepsilon,\infty).  
\end{equation}

We recall that two self-adjoint extensions $S_1$ and $S_2$ of $S$ are called 
{\it relatively prime} (or {\it disjoint}) if $\dom (S_1)\cap\dom (S_2)=\dom (S)$. 
The following result will play a role later on (cf., e.g., \cite[Lemma~2.8]{AGMT10} for an 
elementary proof): 

\begin{lemma}\lb{l2.3}
Suppose Hypothesis~\ref{h2.2}. Then the Friedrichs extension $S_F$ and the Krein--von 
Neumann extension $S_K$ of $S$ are relatively prime, that is,
\begin{equation}\label{utrre}
\dom (S_F)\cap\dom(S_K)=\dom(S). 
\end{equation}
\end{lemma}

Next, we consider a self-adjoint operator $T$ in $\cH$ which is bounded from below, 
that is, $T\geq\alpha I_{\cH}$ for some $\alpha\in\bbR$. We denote by 
$\{E_T(\lambda)\}_{\lambda\in\bbR}$ the family of strongly right-continuous spectral 
projections of $T$, and introduce for $-\infty\leq a<b$,  as usual, 
\begin{equation}\label{74ed}
E_T\big((a,b)\big)=E_T(b_{-})-E_T(a)\quad\text{and}\quad 
E_T(b_{-})=\slim_{\varepsilon\downarrow 0}E_T(b-\varepsilon).
\end{equation}  
In addition, we set 
\begin{equation}\label{i5r}
\mu_{T,j}:=\inf\,\big\{\lambda\in\bbR\,\big|\,
\dim(\ran(E_T((-\infty,\lambda))))\geq j\big\},\quad j\in\bbN.
\end{equation} 
Then, for fixed $k\in\bbN$, either: 
\\
$(i)$ $\mu_{T,k}$ is the $k$th eigenvalue of $T$ counting multiplicity 
below the bottom of the essential spectrum, $\sigma_{ess}(T)$, of $T$, 
\\
or, 
\\
$(ii)$ $\mu_{T,k}$ is the bottom of the essential spectrum of $T$, 
\begin{equation}\label{85f4}
\mu_{T,k}=\inf\,\big\{\lambda\in\bbR\,\big|\,\lambda\in\sigma_{ess}(T)\big\}, 
\end{equation}
and in that case $\mu_{T,k+\ell}=\mu_{T,k}$, $\ell\in\bbN$, and there are at 
most $k-1$ eigenvalues (counting multiplicity) of $T$ below $\mu_{T,k}$. 

We now record a basic result of M. Krein \cite{Kr47} with an extension due 
to Alonso and Simon \cite{AS80} and some additional results recently derived in 
\cite{AGMST10}. For this purpose we introduce the {\it reduced  
Krein--von Neumann operator} $\hatt S_K$ in the Hilbert space 
\begin{equation}\lb{hattH}
\hatt\cH:=\big(\ker(S^*)\big)^{\bot}=\big(\ker(S_K)\big)^{\bot}  
\end{equation}
by 
\begin{align}\lb{2.17}
\hatt{S}_K & :=P_{(\ker(S_K))^{\bot}} S_K|_{(\ker(S_K))^{\bot}},  
\quad\dom(\hatt{S}_K)=\dom S_K\cap\hatt\cH,      
\end{align} 
where $P_{(\ker(S_K))^\bot}$ denotes the orthogonal projection onto $(\ker(S_K))^\bot$.
One then obtains  
\begin{equation}\lb{SKinv}
\big(\hatt{S}_K\big)^{-1}=P_{(\ker(S_K))^{\bot}}(S_F)^{-1}|_{(\ker(S_K))^{\bot}},    
\end{equation}
a relation due to Krein \cite[Theorem~26]{Kr47} (see also \cite[Corollary~5]{Ma92}).

\begin{theorem}\lb{t2.4}
Suppose Hypothesis~\ref{h2.2}. Then
\begin{equation}\lb{Barr-5}
\varepsilon\leq\mu_{S_F,j}\leq\mu_{\hatt S_K,j},   \quad j\in\bbN.
\end{equation} 
In particular, if the Friedrichs extension $S_F$ of $S$ has purely discrete
spectrum, then, except possibly for $\lambda=0$, the Krein--von Neumann extension
$S_K$ of $S$ also has purely discrete spectrum in $(0,\infty)$, that is, 
\begin{equation}\lb{ESSK}
\sigma_{ess}(S_F)=\emptyset\,\text{ implies }\,\sigma_{ess}(S_K)\subseteq\{0\}.      
\end{equation}
In addition, if $p\in (0,\infty]$, then $(S_F-z_0 I_{\cH})^{-1}\in\cB_p(\cH)$ 
for some $z_0\in\bbC\backslash [\varepsilon,\infty)$ implies
\begin{equation}\lb{CPK}
(S_K-zI_{\cH})^{-1}\big|_{(\ker(S_K))^{\bot}}\in\cB_p\big(\hatt \cH\big) 
\,\text{ for all $z\in\bbC\backslash [\varepsilon,\infty)$}.  
\end{equation}
In fact, the $\ell^p(\bbN)$-based trace ideal $\cB_p(\cH)$ 
$\big($resp., $\cB_p\big(\hatt \cH\big)$$\big)$  
of $\cB(\cH)$ $\big($resp., $\cB\big(\hatt \cH\big)$$\big)$ can be 
replaced by any two-sided symmetrically normed ideal of $\cB(\cH)$ 
$\big($resp., $\cB\big(\hatt \cH\big)$$\big)$.
\end{theorem}

We note that \eqref{ESSK} is a classical result of Krein \cite{Kr47}. Apparently, \eqref{Barr-5} 
in the context of infinite deficiency indices was first proven by Alonso and Simon \cite{AS80} 
by a somewhat different method. Relation \eqref{CPK} was proved in \cite{AGMST10}.

Assuming that $S_F$ has purely discrete spectrum, let 
$\{\lambda_{K, j}\}_{j\in\bbN}\subset(0,\infty)$ be the strictly positive eigenvalues 
of $S_K$ enumerated in nondecreasing order, counting multiplicity, and let
\begin{equation} 
N(\lambda,S_K):=\#\{j\in\bbN\,|\,0<\lambda_{K,j} < \lambda\},  \quad \lambda > 0,    \lb{2.22} 
\end{equation}
be the eigenvalue distribution function for $S_K$. Similarly, let 
$\{\lambda_{F, j}\}_{j\in\bbN}\subset(0,\infty)$ denote the eigenvalues 
of $S_F$, again enumerated in nondecreasing order, counting multiplicity, and by 
\begin{equation} 
N(\lambda,S_F):=\#\{j\in\bbN\,|\, \lambda_{F,j} < \lambda\}, 
\quad \lambda > 0,
\end{equation}  
the corresponding eigenvalue counting function for $S_F$. Then inequality 
\eqref{Barr-5} implies
\begin{equation}
N(\lambda,S_K) \leq N(\lambda,S_F), \quad \lambda > 0.   \lb{2.24} 
\end{equation}
In particular, any estimate for the eigenvalue counting function for the Friedrichs extension 
$S_F$, in turn, yields one for the Krein--von Neumann extension $S_K$ 
(focusing on strictly positive eigenvalues of $S_K$ according to \eqref{2.22}). 
While this is a viable approach to estimate the eigenvalue counting function \eqref{2.22} for  
$S_K$, we will proceed along a different route in Section \ref{s3} and directly 
exploit the one-to-one corrspondence between strictly positive eigenvalues of $S_K$ 
and the eigenvalues of its underlying abstract buckling problem to be described next.

To describe the abstract buckling problem naturally associated with the Krein--von Neumann 
extension as described in \cite{AGMST10}, we start by introducing an abstract version of 
\cite[Proposition\ 1]{Gr83} (see \cite{AGMST10} for a proof):

\begin{lemma}\lb{l2.5}
Assume Hypothesis~\ref{h2.2} and let $\lambda\in\bbC\backslash\{0\}$. 
Then there exists some $f\in\dom(S_K)\backslash\{0\}$ with
\begin{equation}\lb{sk1}
S_K f=\lambda f   
\end{equation}
if and only if there exists $w\in\dom(S^* S)\backslash\{0\}$ such that
\begin{equation}\lb{sk2}
S^* Sw=\lambda S w.   
\end{equation}
In fact, the solutions $f$ of \eqref{sk1} are in one-to-one correspondence with the 
solutions $w$ of \eqref{sk2} as evidenced by the formulas
\begin{align}\label{yt5rr}
w & =(S_F)^{-1}S_K f,   \\ 
f & =\lambda^{-1}Sw.   
\end{align}
Of course, since $S_K\geq 0$ is self-adjoint, any $\lambda\in\bbC\backslash\{0\}$ 
in \eqref{sk1} and \eqref{sk2}  necessarily satisfies $\lambda\in(0,\infty)$.
\end{lemma}

It is the linear pencil eigenvalue problem $S^*Sw=\lambda Sw$ in \eqref{sk2} that we call the 
{\it abstract buckling problem} associated with the Krein--von Neumann extension $S_K$ of $S$.

Next, we turn to a variational formulation of the correspondence between the 
inverse of the reduced Krein--von Neumann extension $\hatt{S}_K$ and the abstract 
buckling problem in terms of appropriate sesquilinear forms by following 
\cite{Ko79}--\cite{Ko84} in the elliptic PDE context. This will then lead to an 
even stronger connection between the Krein--von Neumann extension $S_K$ of $S$ 
and the associated abstract buckling eigenvalue problem \eqref{sk2}, culminating 
in the unitary equivalence result in Theorem~\ref{t2.6} below. 

Given the operator $S$, we introduce the following symmetric forms in $\cH$,
\begin{align}\label{tarcd}
\mathfrak{a}(f,g) & :=(Sf,Sg)_{\cH},\quad f,g\in\dom(\mathfrak{a}):=\dom(S),    \\ 
\mathfrak{b}(f,g) & :=(f,Sg)_{\cH},\quad f,g\in\dom(\mathfrak{b}):=\dom(S).    
\end{align}
Then $S$ being densely defined and closed implies that the sesquilinear form $\mathfrak{a}$ 
shares these properties, while $S\geq\varepsilon I_{\cH}$ from Hypothesis~\ref{h2.2} 
implies that  $\mathfrak{a}$ is bounded from below, that is, 
\begin{equation}\lb{2.28}
\mathfrak{a}(f,f)\geq\varepsilon^2\|f\|_{\cH}^2,\quad f\in\dom(S).       
\end{equation} 
(The inequality \eqref{2.28} follows based on the assumption $S\geq\varepsilon I_{\cH}$ 
by estimating $(Sf,Sg)_{\cH}=\big([(S-\varepsilon I_{\cH})+\varepsilon I_{\cH}]f, 
[(S-\varepsilon I_{\cH})+\varepsilon I_{\cH}]g\big)_{\cH}$ from below.)

Thus, one can introduce the Hilbert space 
\begin{equation}\label{itre3er} 
\cW:=\big(\dom(S),(\cdot,\cdot)_{\cW}\big), 
\end{equation} 
with associated scalar product 
\begin{equation} 
(f,g)_{\cW}:=\mathfrak{a}(f,g)=(Sf,Sg)_{\cH},\quad f,g\in\dom(S). 
\end{equation}
In addition, we note that $\iota_{\cW}:\cW\hookrightarrow\cH$, the embedding operator of $\cW$  
into $\cH$, is continuous due to $S\geq\varepsilon I_{\cH}$. Hence, precise notation would be using  
\begin{equation}
(w_1,w_2)_{\cW}=\mathfrak{a}(\iota_{\cW} w_1,\iota_{\cW} w_2) 
=(S\iota_{\cW} w_1,S\iota_{\cW} w_2)_{\cH},\quad w_1,w_2\in\cW,   
\end{equation}
but in the interest of simplicity of notation we will omit the embedding 
operator $\iota_{\cW}$ in the following. 

With the sesquilinear forms $\mathfrak a$ and $\mathfrak b$ and the Hilbert space $\cW$ as 
above, given $w_2\in\cW$, the map $\cW\ni w_1\mapsto (w_1,S w_2)_\cH\in{\mathbb{C}}$ is continuous. 
This allows us to define the operator $Tw_2$ as the unique element in $\cW$ such that 
\begin{equation}\label{ur332}
(w_1,Tw_2)_{\cW}= (w_1,Sw_2)_{\cH}\,\text{ for all }\,w_1\in\cW.
\end{equation} 
This implies
\begin{equation}\label{oi7g4dc4}
\mathfrak{a}(w_1,Tw_2)=(w_1,Tw_2)_{\cW}=(w_1,Sw_2)_{\cH}=\mathfrak{b}(w_1,w_2) 
\end{equation}
for all $ w_1,w_2\in\cW$. In addition, the operator $T$ satisfies  
\begin{equation}\lb{2.33}
0\leq T=T^*\in\cB(\cW)\quad\text{and}\quad\|T\|_{\cB(\cW)}\leq\varepsilon^{-1}.   
\end{equation}
We will call $T$ the {\it abstract buckling problem operator} associated 
with the Krein--von Neumann extension $S_K$ of $S$.  

Next, recalling the notation $\hatt\cH=\big(\ker(S^*)\big)^{\bot}$ (cf. \eqref{hattH}), 
we introduce the operator
\begin{equation}\label{UYBgb}
\hatt{S}:\cW\to\hatt\cH,\quad w\mapsto S w.  
\end{equation}

Clearly, $\ran\big(\hatt{S}\,\big)=\ran(S)$ and since $S\geq\varepsilon I_{\cH}$ 
for some $\varepsilon>0$ and $S$ is closed in $\cH$, $\ran (S)$ is also closed, and 
hence coincides with $\big(\ker(S^*)\big)^\bot$. This yields 
\begin{equation}\label{7hOKI}
\ran\big(\hatt{S}\,\big)=\ran(S)=\hatt\cH.    
\end{equation} 
In fact, it follows that $\hatt{S}\in\cB(\cW,\hatt\cH)$ maps $\cW$ unitarily onto 
$\hatt\cH$ (cf. \cite{AGMST10}).

Continuing, we briefly recall the polar decomposition of $S$, 
\begin{equation}\lb{polar}
S=U_S|S|,  
\end{equation} 
where, with $\varepsilon > 0$ as in Hypothesis~\ref{h2.2}, 
\begin{equation}
|S|=(S^*S)^{1/2}\geq\varepsilon I_{\cH}\,\text{ and }\, 
U_S\in\cB\big(\cH,\hatt\cH\big)\,\text{ unitary.}  
\end{equation}

Then the principal unitary equivalence result proved in \cite{AGMST10} reads as follows: 

\begin{theorem}\lb{t2.6}
Assume Hypothesis~\ref{h2.2}. Then the inverse of the reduced Krein--von Neumann extension 
$\hatt S_K$ in $\hatt\cH$ and the abstract buckling problem operator $T$ in $\cW$ are 
unitarily equivalent. Specifically,
\begin{equation}\lb{11.20}
\big(\hatt{S}_K\big)^{-1}=\hatt{S}T\bigl(\hatt{S}\,\bigr)^{-1}.    
\end{equation}
In particular, the nonzero eigenvalues of $S_K$ are reciprocals of the eigenvalues of $T$.  
Moreover, one has
\begin{equation}\lb{11.20a}
\big(\hatt{S}_K\big)^{-1}=U_S\big[|S|^{-1}S|S|^{-1}\big](U_S)^{-1},    
\end{equation}
where $U_S\in\cB\big(\cH,\hatt\cH\big)$ is the unitary operator in the polar 
decomposition \eqref{polar} of $S$ and the operator $|S|^{-1}S|S|^{-1}\in\cB(\cH)$ 
is self-adjoint and strictly positive in $\cH$. 
\end{theorem}

We emphasize that the unitary equivalence in \eqref{11.20} is independent of any spectral assumptions 
on $S_K$ (such as the spectrum of $S_K$ consists of eigenvalues only) and applies to the restrictions 
of $S_K$ to its pure point, absolutely continuous, and singularly continuous spectral subspaces, 
respectively.

Equation \eqref{11.20a} is motivated by rewriting the abstract linear pencil buckling eigenvalue 
problem \eqref{sk2}, $S^*Sw=\lambda Sw$, $\lambda\in\bbC\backslash\{0\}$, in the form
\begin{equation}\label{uu54434}
|S|^{-1}Sw=(S^*S)^{-1/2}Sw=\lambda^{-1}(S^*S)^{1/2}w=\lambda^{-1}|S|w  
\end{equation}
and hence in the form of a standard eigenvalue problem
\begin{equation}\label{8644}
|S|^{-1}S|S|^{-1}v=\lambda^{-1}v,\quad\lambda\in\bbC\backslash\{0\}, 
\quad v:=|S|w.  
\end{equation}
Again, self-adjointness and strict positivity of $|S|^{-1}S|S|^{-1}$ imply $\lambda\in (0,\infty)$. 

We conclude this section with an elementary result recently noted in \cite{BGMM14} that relates 
the nonzero eigenvalues of $S_K$ directly with the sesquilinear forms $\mathfrak{a}$ and 
$\mathfrak{b}$: 

\begin{lemma}\lb{l2.7}
Assume Hypothesis~\ref{h2.2} and introduce 
\begin{align}\lb{2.42}
& \sigma_p({\mathfrak{a}},{\mathfrak{b}}):=\big\{\lambda\in\bbC\,\big|\,
\text{there exists } \, g_{\lambda}\in\dom(S)\backslash\{0\}     
\nonumber\\ 
& \hspace*{3.00cm} 
\text{such that } \,  
{\mathfrak{a}}(f,g_{\lambda})=\lambda\,{\mathfrak{b}}(f,g_{\lambda}), \quad  f\in\dom(S)\big\}.      
\end{align}
Then
\begin{equation}\lb{2.43}
\sigma_p({\mathfrak{a}},{\mathfrak{b}})=\sigma_p(S_K)\backslash\{0\}     
\end{equation}
{\rm (}counting multiplicity\,{\rm )}, in particular, 
$\sigma_p({\mathfrak{a}},{\mathfrak{b}})\subset(0,\infty)$, and $g_{\lambda}\in\dom(S)\backslash\{0\}$ 
in \eqref{2.42} actually satisfies 
\begin{equation}\lb{2.44}
g_{\lambda}\in\dom(S^*S),\quad S^*Sg_{\lambda}=\lambda Sg_{\lambda}.       
\end{equation}
In addition, 
\begin{equation}\lb{2.45}
\lambda\in\sigma_p({\mathfrak{a}},{\mathfrak{b}})\,\text{ if and only if }\,\lambda^{-1}\in\sigma_p(T)     
\end{equation}
{\rm (}counting multiplicity\,{\rm )}. Finally, 
\begin{equation}\lb{2.46}
T\in\cB_{\infty}(\cW)\,\Longleftrightarrow\,\big(\hatt{S}_K\big)^{-1}\in\cB_{\infty}\big(\hatt\cH\big) 
\,\Longleftrightarrow\,\sigma_{ess}(S_K)\subseteq\{0\}, 
\end{equation}
and hence, 
\begin{equation}\lb{2.47}
\sigma_p({\mathfrak{a}},{\mathfrak{b}})=\sigma(S_K)\backslash\{0\}=\sigma_d(S_K)\backslash\{0\}     
\end{equation}
if \eqref{2.46} holds. In particular, if one of $S_F$ or $|S|$ has purely discrete spectrum 
$($i.e., $\sigma_{ess}(S_F)=\emptyset$ or $\sigma_{ess}(|S|)=\emptyset$$)$, then \eqref{2.46} and 
\eqref{2.47} hold. 
\end{lemma}
\begin{proof}
We begin by noting that \eqref{2.28} and the fact that 
${\mathfrak{b}}(f,f)\geq\varepsilon\|f\|_{\cH}^2$
imply $\sigma_p({\mathfrak{a}},{\mathfrak{b}})\subset(0,\infty)$. 
Moreover, using the fact that the self-adjoint operator in $\cH$ 
uniquely associated with the form $\mathfrak{a}$ is given by 
$S^*S$ (cf. \cite[Example~VI.2.13]{Ka80}), and that 
$\mathfrak{a}(f,g_{\lambda}) = \lambda\,{\mathfrak{b}}(f,g_{\lambda}) 
= \lambda (f,Sg_{\lambda})_{\cH}$, $f\in\dom(\mathfrak{a})=\dom(S)$, 
the first representation theorem for quadratic forms (cf. \cite[Theorem~VI.2.1\,(iii)]{Ka80}) 
implies \eqref{2.44}. An application of Lemma~\ref{l2.5} then yields \eqref{2.43}. 
Relation \eqref{2.45} then follows from \eqref{2.43} and \eqref{11.20}. The first 
equivalence in \eqref{2.46} again is a consequence of \eqref{11.20} and the fact that 
$\hatt{S}$ maps $\cW$ unitarily onto $\hatt\cH$; the second equivalence in \eqref{2.46} 
follows from \eqref{2.17}. The final claim in Lemma~\ref{l2.7} involving discrete spectra of $S_F$ 
or $|S|$ is a consequence of \eqref{ESSK} or \eqref{11.20a} and the equivalence statements 
in \eqref{2.46}.  
\end{proof}

One notices that $f\in\dom(S)$ in the definition \eqref{2.42} of 
$\sigma_p({\mathfrak{a}},{\mathfrak{b}})$ can be replaced by $f\in C(S)$ for 
any (operator) core $C(S)$ for $S$ (equivalently, by any form core for the form $\mathfrak{a}$).

\section{An Upper Bound for the Eigenvalue Counting Function for Higher-Order Krein 
Laplacians on Finite Volume Domains} 
\lb{s3}

In this section we derive an upper bound for the eigenvalue counting function for higher-order 
Krein Laplacians on open, nonempty domains $\Omega \subset \bbR^n$ of finite (Euclidean) 
volume. In particular, no assumptions on the boundary of $\Omega$ will be made.

Before introducing the class of constant coefficient partial differential operators in $L^2(\Omega)$ 
at hand, we recall a few auxiliary facts to be used in the proof of Theorem \ref{t3.9}.

\begin{lemma} \lb{l3.1}
Suppose that $S$ is a densely defined, symmetric, closed operator in $\cH$. Then $|S|$ and hence 
$S$ is infinitesimally bounded with respect to $S^* S$, more precisely, one has 
\begin{align}
\begin{split} 
\text{for all $\varepsilon > 0$, } \, \|S f\|_{\cB(\cH)} = \| |S| f\|_{\cB(\cH)} \leq 
\varepsilon \|S^* S f\|_{\cH}^2 + (4 \varepsilon)^{-1} \|f\|_{\cH}^2,&     \lb{3.1} \\
f \in \dom(S^* S).&         
\end{split} 
\end{align}
In addition, $S$ is relatively compact with respect to $S^* S$ if $|S|$, or equivalently, $S^* S$, has 
compact resolvent. In particular,
\begin{equation}
\sigma_{ess}(S^* S - \lambda S) = \sigma_{ess}(S^* S), \quad \lambda \in \bbR. 
\end{equation}
\end{lemma}
\begin{proof}
Employing the polar decomposition of $S$, $S = U_S |S|$, where $U_S$ is a partial isometry and 
$|S| = (S^* S)^{1/2}$ (cf.\ \cite[Sect.\ VI.2.7]{Ka80}), one obtains 
\begin{equation} 
\|S f\|_{\cB(\cH)} = \| |S| f\|_{\cB(\cH)}, \quad f \in \dom(S) = \dom(|S|), 
\end{equation} 
and hence the spectral theorem 
applied to $|S|$, together with the elementary inequality $\lambda \leq \varepsilon \lambda^2 
+ (4 \varepsilon)^{-1}$, $\varepsilon > 0$, $\lambda \geq 0$, proves inequality \eqref{3.1}. 

The relative compactness assertion then follows from
\begin{equation}
S (S^*S + I_{\cH})^{-1} = \Big[S \big(|S|^2 + I_{\cH}\big)^{-1/2}\Big] 
\big(|S|^2 + I_{\cH}\big)^{-1/2} \in \cB_{\infty}(\cH), 
\end{equation} 
since $S \big(|S|^2 + I_{\cH}\big)^{-1/2} \in \cB(\cH)$ and 
$\big(|S|^2 + I_{\cH}\big)^{-1/2} \in \cB_{\infty}(\cH)$.  
\end{proof}

Given a lower semibounded, self-adjoint operator $T \geq c_T I_{\cH}$ in $\cH$, we denote 
by $q_T$ its uniquely associated form, that is, 
\begin{equation}
\gq_T(f,g) = \big(|T|^{1/2} f, \sgn(T) |T|^{1/2} \big)_{\cH},    \quad  
f, g \in \dom(\gq) = \dom \big(|T|^{1/2}\big),  
\end{equation} 
and by $\{E_T(\lambda)\}_{\lambda \in \bbR}$ the family of spectral projections of $T$. 
We recall the following well-known variational characterization of dimensions of spectral 
projections $E_T([c_T, \mu))$, $\mu > c_T$.  

\begin{lemma} \lb{l3.2}
Assume that $c_T I_{\cH} \leq T$ is self-adjoint in $\cH$ and $\mu > c_T$. Suppose that 
$\cF \subset \dom \big(|T|^{1/2}\big)$ is a linear subspace such that 
\begin{equation}
\gq_T(f,f) < \mu \|f\|_{\cH}^2, \quad f \in \cF\backslash\{0\}.
\end{equation}  
Then,
\begin{equation}
\dim \big(\ran(E_T([c_T,\mu)))\big) = \sup_{\cF \subset \dom (|T|^{1/2})} (\dim\,(\cF)). 
\end{equation}
\end{lemma}

We add the following elementary observation: Let 
$c\in\bbR$ and $B\geq c I_{\cH}$ be a self-adjoint operator in $\cH$, and introduce the sesquilinear form $b$ in $\cH$ associated with $B$ via
\begin{align}
\begin{split}
& b(u,v) = \big((B - c I_{\cH})^{1/2} u, (B - c I_{\cH})^{1/2} v\big)_{\cH}
+ c (u,v)_{\cH}, \\  
& u,v \in \dom(b) = \dom\big(|B|^{1/2}\big).    \lb{B.57}
\end{split}
\end{align}
Given $B$ and $b$, one introduces the Hilbert space $\cH_b \subseteq \cH$ by 
\begin{align}
& \cH_b =\big(\dom\big(|B|^{1/2}\big), (\cdot,\cdot)_{\cH_b}\big),   \no \\ 
& (u,v)_{\cH_b} =  b(u,v) + (1-c) (u,v)_{\cH}   \lb{B.58} \\
& \hspace*{1.2cm} = \big((B - c I_{\cH})^{1/2} u, (B - c I_{\cH})^{1/2} v\big)_{\cH} 
+ (u,v)_{\cH}    \no \\
& \hspace*{1.2cm} = \big((B + (1-c) I_{\cH})^{1/2} u, (B + (1-c) I_{\cH})^{1/2} v\big)_{\cH}. 
\no 
\end{align}
One observes that 
\begin{equation} 
(B + (1 - c)I_{\cH})^{1/2} \colon \cH_b \to \cH \, \text{ is unitary.}     \lb{B.59}
\end{equation}

\begin{lemma} [see, e.g., \cite{GM09}] \lb{lB.2} 
Let $\cH$, $B$, $b$, and $\cH_b$ be as in \eqref{B.57}--\eqref{B.59}. Then $B$ has purely discrete spectrum, that is, $\sigma_{\rm ess} (B) = \emptyset$, if and only if $\cH_b$ embeds compactly
into $\cH$. 
\end{lemma}

Next we turn to higher-order Laplacians $(- \Delta)^m$ in $L^2(\Omega)$ and 
hence introduce the following assumptions on $\Omega \subset \bbR^n$, $n \in \bbN$. 

\begin{hypothesis} \lb{h3.3}
Let $n \in \bbN$ and assume that $\emptyset \neq \Omega \subset \bbR^n$ is an open set of finite 
$($Euclidean$)$ volume $($denoted by $|\Omega| < \infty$$)$.  
\end{hypothesis} 

The above hypothesis is going to be relevant for the validity of the compact embedding 
of $\mathring W^1(\Omega)$ into $L^2(\Omega)$. Necessary and sufficient conditions for this compact 
embedding to hold in terms of appropriate capacities can be found, for instance, in \cite{Ad70}, 
\cite[Ch.\ 6]{AF03}, \cite[Ch.\ 6]{Ma11}. Since we seek such an embedding for finite-volume domains, 
and the precise statement appears to be difficult to discern from the existing 
literature on this subject, we decided to spell out the details for the convenience of the reader. 
In fact, for completeness, we will discuss a more general result in connection with $L^p$-based 
Sobolev spaces which we introduce next. 

Suppose that $\emptyset \neq \Omega \subseteq \bbR^n$ is open and, for each $p\in[1,\infty]$
and $k\in{\mathbb{N}}$, define 
\begin{equation}\label{eq:EMB.A}
W^{k,p}(\Omega):= \big\{u\in L^1_{\rm loc}(\Omega) \,\big| \,\partial^\alpha u\in L^p(\Omega)
\text{ if } 0 \leq |\alpha|\leq k\big\},
\end{equation}
where the derivatives are taken in the sense of distributions. This 
becomes a Banach space when endowed with the natural norm, 
\begin{equation}\label{eq:EMB.B}
\|u\|_{W^{k,p}(\Omega)}:=\sum_{0 \leq |\alpha|\leq k}\|\partial^\alpha u\|_{L^p(\Omega)} 
\end{equation} 
(see also \eqref{3.8}, but for the purpose at hand we now prefer to use \eqref{eq:EMB.B}). 
In the same setting we also consider the closed linear subspace 
$\mathring{W}^{k,p}(\Omega)$ of $W^{k,p}(\Omega)$ given by 
\begin{equation}\label{eq:EMB.C} 
\mathring{W}^{k,p}(\Omega):=\ol{C_0^{\infty}(\Omega)}^{W^{k,p}(\Omega)}.
\end{equation}
A useful observation, seen directly from definitions, is that whenever $\alpha\in{\mathbb{N}}_0^n$ 
is such that $|\alpha|\leq k-1$ then 
\begin{equation}\label{eq:OP}
\begin{array}{c}
\partial^\alpha:\mathring{W}^{k,p}(\Omega)\longrightarrow \mathring{W}^{k-|\alpha|,p}(\Omega)
\,\,\mbox{ is a well-defined,}
\\[4pt] 
\mbox{linear, and bounded operator, with norm $\leq 1$}. 
\end{array}
\end{equation}

Our goal is to prove the following compact embedding result. 

\begin{theorem}\label{t3.7}
Assume Hypothesis~\ref{h3.3} and pick $p\in[1,\infty]$. Let $p_\ast$ be an arbitrary number 
in $[1,\infty)$ if $p\geq n$, and suppose that $p_\ast\in\big[1,\frac{np}{n-p}\big)$ if $1\leq p<n$. Then 
\begin{equation}\label{eq:EMB.D}
\mathring{W}^{1,p}(\Omega)\hookrightarrow L^{p_\ast}(\Omega) \,\text{ compactly}.
\end{equation}
\end{theorem}
\begin{proof}
To start the proof, we note that since $\Omega$ has finite measure, the 
scale of Lebesgue spaces in $\Omega$ is nested. Specifically, H\"older's inequality implies 
\begin{equation}\label{eq:REdTu}
L^{q_1}(\Omega)\hookrightarrow L^{q_2}(\Omega)\, \text{ continuously if $0<q_2\leq q_1\leq\infty$.}
\end{equation}

We continue by recalling a useful result from \cite{CRW13}. 
Given any $p_1\in(1,\infty]$ and $p_2\in[1,\infty)$, define the space 
\begin{equation}\label{eq:EMB.1}
E^{p_1,p_2}(\Omega):= \big\{u\in L^{p_1}(\Omega) \, \big| \, 
\partial_j u\in L^{p_2}(\Omega)\,\mbox{ for }\,1\leq j\leq n\big\},
\end{equation}
and equip it with the natural norm 
\begin{equation}\label{eq:EMB.2}
\|u\|_{E^{p_1,p_2}(\Omega)}:=\|u\|_{L^{p_1}(\Omega)}+\sum_{j=1}^n\|\partial_j u\|_{L^{p_2}(\Omega)}.
\end{equation}
Then a particular version of \cite[Theorem~2.2, p.\,27]{CRW13} implies that under the 
current assumptions on $\Omega$, 
\begin{equation}\label{eq:EMB.3}
E^{p_1,p_2}(\Omega)\hookrightarrow L^{p_3}(\Omega) \,\text{ compactly, for each } \, p_3\in[1,p_1).
\end{equation}

Next, we denote by tilde the operator of extension by zero of functions defined in $\Omega$ to the entire 
Euclidean space ${\mathbb{R}}^n$. Since for every $\varphi\in C_0^{\infty}(\Omega)$ we have 
that $\widetilde{\varphi}\in C_0^{\infty}({\mathbb{R}}^n)\subset W^{1,p}({\mathbb{R}}^n)$ and 
$\|\widetilde{\varphi}\|_{W^{1,p}(\mathbb{R}^n)}=\|\varphi\|_{W^{1,p}(\Omega)}$, it follows
from \eqref{eq:EMB.C} that tilde induces a mapping 
\begin{equation}\label{eq:Tbav}
\mathring{W}^{1,p}(\Omega)\ni u\mapsto\widetilde{u}\in W^{1,p}(\mathbb{R}^n)
\,\text{ which is a linear isometry.}  
\end{equation}
Bearing this in mind, in the case when $1\leq p<n$, for each 
$u\in\mathring{W}^{1,p}(\Omega)$ we may use the classical Sobolev embedding theorem 
(in ${\mathbb{R}}^n$) in order to estimate 
\begin{equation}\label{eq:uGVV}
\|u\|_{L^{\frac{np}{n-p}}(\Omega)}=\|\widetilde{u}\|_{L^{\frac{np}{n-p}}(\mathbb{R}^n)}
\leq C_{n,p}\|\widetilde{u}\|_{W^{1,p}(\mathbb{R}^n)}=C_{n,p}\|u\|_{W^{1,p}(\Omega)},
\end{equation}
for some finite constant $C_{n,p}>0$. This proves that 
\begin{equation}\label{eq:uGVV.2}
\mathring{W}^{1,p}(\Omega)\hookrightarrow L^{\frac{np}{n-p}}(\Omega)
\,\,\mbox{ continuously, if }\,\,1\leq p<n.
\end{equation}
We divide the remaining portion of the proof into two cases. \\[1mm] 
\noindent
{\bf Case~1:} {\it We claim that \eqref{eq:EMB.D} holds 
when $1\leq p<n$ and $p_\ast\in\big[1,\frac{np}{n-p}\big)$.}
In this scenario, pick $p_1\in\big(p_\ast,\frac{np}{n-p}\big]$ and $p_2\in[1,p]$. 
This choice entails
\begin{equation}\label{eq:EMB.3BB}
E^{p_1,p_2}(\Omega)\hookrightarrow L^{p_\ast}(\Omega) \,\text{ compactly,}
\end{equation}
by \eqref{eq:EMB.3}, and
\begin{equation}\label{eq:uGVV.2BB}
\mathring{W}^{1,p}(\Omega)\hookrightarrow E^{p_1,p_2}(\Omega) \, \text{ continuously,}
\end{equation}
by \eqref{eq:uGVV.2} and (suitable applications of) \eqref{eq:REdTu}.
Collectively, \eqref{eq:EMB.3BB} and \eqref{eq:uGVV.2BB} yield \eqref{eq:EMB.D} in this 
case. \\[1mm] 
\noindent
{\bf Case~2:} {\it We claim that \eqref{eq:EMB.D} holds 
when $n\leq p\leq\infty$ and $p_\ast\in [1,\infty)$.}
To justify this claim, pick an arbitrary $q\in[1,n)$. 
In particular, $q<p$, so \eqref{eq:REdTu} yields that on one hand, 
\begin{equation}\label{eq:uGVV.CC}
\mathring{W}^{1,p}(\Omega)\hookrightarrow \mathring{W}^{1,q}(\Omega) \, \text{ continuously.} 
\end{equation}
On the other hand, by what has already proved in Case~1, 
\begin{equation}\label{eq:uGVV.CC2fr}
\mathring{W}^{1,q}(\Omega)\hookrightarrow L^{q_\ast}(\Omega) \,\text{ compactly, for each } \,
q_\ast\in\big[1,\tfrac{nq}{n-q}\big).
\end{equation}
Combining \eqref{eq:uGVV.CC} with \eqref{eq:uGVV.CC2fr} and observing that 
$\frac{nq}{n-q}\nearrow\infty$ as $q\nearrow n$, one ultimately deduces that \eqref{eq:EMB.D} 
holds in this case as well. 
\end{proof}

Theorem~\ref{t3.7} has two notable consequences, recorded below.
The first such corollary deals with the following compactness result.

\begin{corollary}\label{c3.8}
Assume Hypothesis~\ref{h3.3}. Then for each $k\in{\mathbb{N}}$ and $p\in[1,\infty)$,  
it follows that 
\begin{equation}\label{eq:EMB.DcN}
\mathring{W}^{k,p}(\Omega)\hookrightarrow L^p(\Omega) \, \text{ compactly.} 
\end{equation}
\end{corollary}
\begin{proof}
This is an immediate consequence of Theorem~\ref{t3.7}, keeping in mind that 
$\mathring{W}^{k,p}(\Omega)\hookrightarrow\mathring{W}^{1,p}(\Omega)$ continuously
and that $p<\frac{np}{n-p}$ when $1\leq p<n$.
\end{proof}

The second corollary of Theorem~\ref{t3.7} alluded to earlier deals with a Poincar\'e-type inequality.

\begin{corollary}\label{c3.8PPP}
Assume Hypothesis~\ref{h3.3}. Then for each $k\in{\mathbb{N}}$ and $p\in[1,\infty)$,  
there exists a constant $C\in(0,\infty)$ with the property that the following 
Poincar\'e-type inequality holds:
\begin{equation}\label{3.10}
\sum_{0\leq|\beta|\leq k-1}\|\partial^\beta u\|_{L^p(\Omega)}\leq C
\sum_{|\alpha|=k}\|\partial^\alpha u\|_{L^p(\Omega)},\quad u\in\mathring{W}^{k,p}(\Omega).
\end{equation}
\end{corollary}
\begin{proof}
We shall prove \eqref{3.10} by induction on $k\in{\mathbb{N}}$. 
\\[1mm] 
\noindent
{\bf Step~1:} {\it We claim that \eqref{3.10} holds when $k=1$, that is, there exists 
$C\in(0,\infty)$ such that}
\begin{equation}\label{3:10.b1}
\|u\|_{L^p(\Omega)}\leq C\|\nabla u\|_{[L^p(\Omega)]^n},\quad u\in\mathring{W}^{1,p}(\Omega).
\end{equation}
Seeking a contradiction, assume that there exists a sequence 
$\{u_j\}_{j\in{\mathbb{N}}}\subset\mathring{W}^{1,p}(\Omega)$ with the property that 
\begin{equation}\label{3:10.b2}
\|u_j\|_{L^p(\Omega)}>j\,\|\nabla u_j\|_{[L^p(\Omega)]^n},\quad j\in\mathbb{N}.
\end{equation}
For each $j\in{\mathbb{N}}$ define 
\begin{equation}\label{eq:nb56}
v_j:=\frac{u_j}{\|u_j\|_{L^p(\Omega)}}\in\mathring{W}^{1,p}(\Omega).
\end{equation}
Note that 
\begin{equation}\label{eq:nb56.a}
\|v_j\|_{L^p(\Omega)}=1 \,\text{ for every } \, j\in{\mathbb{N}},
\end{equation}
and $\nabla v_j=\nabla u_j/\|u_j\|_{L^p(\Omega)}$ which, in light of \eqref{3:10.b2}, implies 
\begin{equation}\label{eq:nb56.b}
\|\nabla v_j\|_{L^p(\Omega)}=\frac{\|\nabla u_j\|_{L^p(\Omega)}}{\|u_j\|_{L^p(\Omega)}}<j^{-1}
\, \text{ for every } \, j\in{\mathbb{N}}.
\end{equation}
From \eqref{eq:nb56}--\eqref{eq:nb56.b} it follows that 
$\{v_j\}_{j\in{\mathbb{N}}}$ is a bounded sequence in $\mathring{W}^{1,p}(\Omega)$.
Granted this fact, Corollary~\ref{c3.8} applies and yields the existence of a strictly increasing 
sequence $\{j_\ell\}_{\ell\in{\mathbb{N}}}\subseteq{\mathbb{N}}$ along with some function 
$v\in L^p(\Omega)$ with the property that
\begin{equation}\label{eq:nb56.c}
v_{j_\ell} \underset{\ell\to\infty}{\longrightarrow} v \, \mbox{ in } \, L^p(\Omega).
\end{equation}
As a consequence of this and \eqref{eq:nb56.a} we deduce that 
\begin{equation}\label{eq:nb56.d}
\|v\|_{L^p(\Omega)}=1.
\end{equation}
Next, we recall that tilde denotes the operator of extension by zero 
of functions defined in $\Omega$ to the entire Euclidean space ${\mathbb{R}}^n$.
In particular, in the sense of distributions, 
\begin{equation}\label{eq:heeww}
\partial_m \widetilde{w}=\widetilde{\partial_m w} \, \mbox{ for each } \,
w\in\mathring{W}^{1,p}(\Omega) \, \text{ and } \, m \in\{1,\dots,n\}.
\end{equation}
Note that $\widetilde{v_j}\in W^{1,p}(\bbR^n)$ for each $j\in{\mathbb{N}}$, 
$\widetilde{v}\in L^p(\bbR^n)$, and $\widetilde{v_{j_\ell}}\to\widetilde{v}$ 
in $L^p(\bbR^n)$ as $\ell\to\infty$. As such, for each test function 
$\phi\in C^\infty_0({\mathbb{R}}^n)$ and each $m \in\{1,\dots,n\}$ we may write 
\begin{align}\label{yrddd}
& \big|{}_{{\mathcal{D}}'({\mathbb{R}}^n)}\big\langle\partial_m \widetilde{v},\phi\big\rangle
_{{\mathcal{D}}({\mathbb{R}}^n)}\big| =
\big|{}_{{\mathcal{D}}'({\mathbb{R}}^n)}\big\langle\widetilde{v},\partial_m \phi\big\rangle
_{{\mathcal{D}}({\mathbb{R}}^n)}\big|
=\bigg|\int_{{\mathbb{R}}^n}\widetilde{v}(x)(\partial_m \phi)(x)\,d^n x\bigg|
\no \\[6pt]
& \quad =\bigg|\lim_{\ell\to\infty}\int_{{\mathbb{R}}^n}
\widetilde{v_{j_\ell}}(x)(\partial_m \phi)(x)\,d^n x\bigg|
=\bigg|\lim_{\ell\to\infty}\int_{{\mathbb{R}}^n}\widetilde{(\partial_m v_{j_\ell})}(x)\phi(x)\,d^n x\bigg|
\nonumber\\[6pt]
& \quad \leq \limsup_{\ell\to\infty}
\Big\|\widetilde{(\partial_m v_{j_\ell})}\Big\|_{L^p({\mathbb{R}}^n)}
\|\phi\|_{L^{p'}({\mathbb{R}}^n)}
\nonumber\\[6pt]
& \quad \leq\|\phi\|_{L^{p'}({\mathbb{R}}^n)}\limsup_{\ell\to\infty}
\big\|\nabla v_{j_\ell}\big\|_{L^p(\Omega)}=0,
\end{align}
by \eqref{eq:heeww}, H\"older's inequality (with $p'$ denoting 
the conjugate exponent of $p$), and \eqref{eq:nb56.b}. Here, and elsewhere,
${}_{\cD(\Omega)'} \langle \, \cdot \, , \, \cdot \, \rangle_{\cD(\Omega)}$
is the standard distributional pairing, with $\cD(\Omega) := C_0^{\infty}(\Omega)$ 
equipped with the usual inductive limit topology.

In turn, the estimate \eqref{yrddd} proves (cf., e.g., \cite[Ch.\ 2]{Mi13}) 
that there exists a constant $c\in{\mathbb{R}}$ such that $\widetilde{v}=c$ a.e. in ${\mathbb{R}}^n$.
In fact, from \eqref{eq:nb56.d} we see that, necessarily,
$\widetilde{v}=|\Omega|^{-1/p}$ a.e.~in ${\mathbb{R}}^n$ which then contradicts the fact that
$\widetilde{v}=0$ in ${\mathbb{R}}^n\setminus\Omega$, given that $|{\mathbb{R}}^n\setminus\Omega|=\infty$.
This contradiction establishes \eqref{3:10.b1} and finishes the proof of Step~1.
\\[1mm] 
\noindent
{\bf Step~2:} {\it We claim that if \eqref{3.10} holds for some $k\in{\mathbb{N}}$, then 
its version written for $k+1$ in place of $k$ is also true.} 
To see that this is the case, assume that $k$ is as above and 
pick an arbitrary $u\in\mathring{W}^{k+1,p}(\Omega)$. Since for each $j\in\{1,\dots,n\}$,  
\eqref{eq:OP} implies that $\partial_j u\in\mathring{W}^{k,p}(\Omega)$, with
$\|\partial_j u\|_{W^{k,p}(\Omega)}\leq\|u\|_{W^{k+1,p}(\Omega)}$, the induction hypothesis
guarantees the existence of some $C\in(0,\infty)$ independent of $u$ such that
\begin{equation}\label{3:10.AI.35}
\sum_{0\leq|\beta|\leq k-1}\|\partial^\beta(\partial_j u)\|_{L^p(\Omega)}\leq C
\sum_{|\alpha|=k}\|\partial^\alpha(\partial_j u)\|_{L^p(\Omega)}.
\end{equation}
Summing over $j\in\{1,\dots,n\}$ and adjusting notation then yields 
\begin{equation}\label{3:10.AI.36}
\sum_{1\leq|\gamma|\leq k}\|\partial^\gamma u\|_{L^p(\Omega)}\leq C
\sum_{|\alpha|=k+1}\|\partial^\alpha u\|_{L^p(\Omega)},
\end{equation}
for a possibly different constant $C\in(0,\infty)$ which is nonetheless independent of $u$. 
Together with \eqref{3:10.b1} this proves that 
\begin{equation}\label{3:10.AI.37}
\sum_{0\leq|\gamma|\leq k}\|\partial^\gamma u\|_{L^p(\Omega)}\leq C
\sum_{|\alpha|=k+1}\|\partial^\alpha u\|_{L^p(\Omega)},\quad u\in \mathring{W}^{k+1,p}(\Omega).
\end{equation}
This completes the treatment of Step~2 and hence finishes the proof.
\end{proof}

In the remainder of the paper we are going to concern ourselves exclusively 
with the $L^2$-based Sobolev space $W^{k,2}(\Omega)$. As such, we agree to 
drop the dependence on the integrability exponent and simply write $W^{k}(\Omega)$.
Hence, 
\begin{equation}
W^k(\Omega)= \big\{u \in L^1_{\loc}(\Omega) \,|\, \partial^{\alpha} u \in L^2(\Omega)  
\text{ if } 0 \leq |\alpha| \leq k\big\}, 
\quad k \in \bbN_0,  
\end{equation}
with $\alpha \in \bbN_0^n$ and $\partial^{\alpha} u$ denoting weak derivatives of $u$. 
The space $W^k(\Omega)$ is endowed with the norm 
\begin{equation}
\|u\|_{k, \Omega} 
=\sum_{0 \leq |\alpha| \leq k} \big\|\partial^{\alpha} u \big\|_{L^2(\Omega)}, 
\quad u \in W^k(\Omega).     \lb{3.8}
\end{equation}   
In addition, define  
\begin{equation}
\mathring W^k (\Omega) = \ol{C_0^{\infty} (\Omega)}^{W^k (\Omega)}, \quad k \in \bbN_0,  \lb{3.9} 
\end{equation}
and note that $\mathring W^k (\Omega)$ is a closed linear subspace of $W^k(\Omega)$. 
Granted Hypothesis \ref{h3.3}, Corollary~\ref{c3.8PPP} then implies the 
Poincar\'e-type inequality 
\begin{equation}
\interleave u \, \interleave_{\ell,\Omega} \leq C 
\interleave u \, \interleave_{k,\Omega}, 
\quad u \in \mathring W^p (\Omega), \; 
\ell \in \bbN_0, \; k \in \bbN, \; \ell \leq k,      \lb{3.10bbbbbbb} 
\end{equation}
where we introduced the abbreviation
\begin{equation}
\interleave u \, \interleave_{k,\Omega}:=
\sum_{|\alpha| = k} \big\|\partial^{\alpha} u \big\|_{L^2(\Omega)}, 
\quad u \in W^k(\Omega), \; k \in \bbN_0. 
\end{equation}
Thus, $\interleave u \, \interleave_{k,\Omega}$, $u \in W^k(\Omega)$, represents an equivalent 
norm on $\mathring W^k (\Omega)$. 

We proceed with the following useful identity:
 
\begin{lemma} \lb{l3.4}
Let $k\in\bbN$ and assume $\emptyset \neq \Omega \subseteq \bbR^n$ is open. Then, 
\begin{equation} \lb{planch}
\sum_{|\alpha |=k}\frac{k!}{\alpha!}\int_\Omega
\big|\big(\partial^\alpha \phi \big)(x)\big|^2\,d^nx =\int_\Omega 
\ol{\big(\Delta^k \phi\big)(x)} \, \phi(x)\,d^nx, \quad \phi\in C^\infty_0(\Omega). 
\end{equation}
\end{lemma}
\begin{proof}
Pick an arbitrary $\phi\in C^\infty_0(\Omega)$.
Using the fact that $\supp \, (\phi) \subset\Omega$ and employing the Plancherel identity, 
one obtains  
\begin{equation} \lb{PL-1}
\int_\Omega \ol{\big(\Delta^k \phi\big)(x)} \, \phi(x)\,d^n x=
\int_{\mathbb{R}^n} \ol{\big(\Delta^k\big)\phi(x)} \, \phi(x)\,d^n x
=\int_{\mathbb{R}^n}|\xi|^{2k} \big|\hatt{\phi}(\xi)\big|^2\,d^n \xi.  
\end{equation} 
Similarly, 
\begin{align} \lb{PL-2}
\begin{split} 
\sum_{|\alpha|=k}\frac{k!}{\alpha!}\int_\Omega \big|\big(\partial^\alpha \phi\big)(x)\big|^2\,d^nx
& = \sum_{|\alpha|=k}\frac{k!}{\alpha!}\int_{\mathbb{R}^n}
\big|\big(\partial^\alpha \phi\big)(x)\big|^2\,d^nx   \\
& = \sum_{|\alpha|=k} \frac{k!}{\alpha!} \int_{\mathbb{R}^n} \xi^{2\alpha} 
\big|\hatt{\phi}(\xi)\big|^2 \, d^n \xi.   
\end{split}
\end{align} 
Since in general, $\Big(\sum\limits_{j=1}^{n} x_j \Big)^N=
\displaystyle\sum_{|\alpha|=N}\frac{|\alpha|!}{\alpha!} \, x^\alpha$, with
$x:=(x_1,...,x_n)$, by the Multinomial Theorem, one concludes that 
\begin{equation} \lb{PL-3}
\sum_{|\alpha|=k}\frac{k!}{\alpha!} \, \xi^{2\alpha}=
\bigg(\sum^n_{j=1} \xi^2_j\bigg)^k = |\xi|^{2k}, \quad \xi \in \bbR^n.
\end{equation} 
Therefore, using \eqref{PL-1}--\eqref{PL-3}, we may write 
\begin{equation} \lb{PL-4}
\sum_{|\alpha|=k}\frac{k!}{\alpha!}\int_\Omega \big|\big(\partial^\alpha \phi\big)(x)\big|^2\,d^nx
=\int_{\mathbb{R}^n}|\xi|^{2k}|\hatt{\phi}(\xi)|^2\,d^n \xi
=\int_\Omega \ol{\big(\Delta^k \phi\big)(x)} \, \phi(x) \,d^nx,
\end{equation}
completing the proof of \eqref{planch}.
\end{proof}

Lemma \ref{l3.4} is a key input for the next result. 

\begin{theorem} \lb{l3.5} 
Assume Hypothesis \ref{h3.3} and let $m \in \bbN$. Consider the minimal operator 
\begin{equation}
A_{min,\Omega,m} := (- \Delta)^m, \quad \dom (A_{min,\Omega,m}) := C_0^{\infty}(\Omega),    
\end{equation}
in $L^2(\Omega)$. Then the closure of $A_{min,\Omega,m}$ in $L^2(\Omega)$ is given by 
\begin{equation}
A_{\Omega,m} = (- \Delta)^m, \quad \dom (A_{\Omega,m}) = \mathring W^{2m} (\Omega).   
\end{equation} 
In addition, $A_{\Omega,m}$ is a strictly positive operator, that is, there exists $\varepsilon > 0$ 
such that 
\begin{equation} 
A_{\Omega,m} \geq \varepsilon I_{\Omega}.
\end{equation} 
\end{theorem}
\begin{proof}
Clearly $A_{min,\Omega,m}$ is symmetric and hence closable. 
Assuming $\phi \in C_0^{\infty}(\Omega)$, repeatedly integrating by parts and an 
application of Lemma \ref{l3.4} yield,   
\begin{align}
\int_{\Omega} \big|\big((- \Delta)^m \phi\big)(x)\big|^2 \, d^n x 
&= \int_{\Omega} \ol{\big((- \Delta)^m \phi\big)(x)} \, \big((- \Delta)^m \phi\big)(x) \, d^n x  \no \\
&=  \int_{\Omega} \ol{\big((- \Delta)^{2m} \phi\big)(x)} \, \phi(x) \, d^n x  \no \\ 
&= \sum_{|\alpha| = 2m} \f{(2m)!}{\alpha!} 
\int_{\Omega} \big|\big(\partial^\alpha \phi\big)(x)\big|^2\,d^nx.      \lb{3.17} 
\end{align}
By density of $C_0^{\infty}(\Omega)$ in $\mathring W^{2m} (\Omega)$, identity \eqref{3.17} extends to 
\begin{equation}
\int_{\Omega} \big|\big((- \Delta)^m u\big)(x)\big|^2 \, d^n x = \sum_{|\alpha| = 2m} \f{(2m)!}{\alpha!} 
\int_{\Omega} \big|\big(\partial^\alpha u\big)(x)\big|^2\,d^nx, \quad u \in \mathring W^{2m} (\Omega). 
\lb{3.18}
\end{equation}
Next, combining the Poincar\'e inequality \eqref{3.10} with \eqref{3.18} implies that for some 
constant $C_{m,\Omega} > 0$, 
\begin{equation}
\int_{\Omega} \big|\big((- \Delta)^m u\big)(x)\big|^2 \, d^n x \geq C_{m,\Omega} 
\sum_{0 \leq |\beta| \leq 2m} \big\|\partial ^{\beta} u\big\|_{L^2(\Omega)}^2 
\approx \|u\|_{m,\Omega}^2, \quad u \in \mathring W^{2m} (\Omega).     \lb{3.19} 
\end{equation}
Finally, consider $\{f_j\}_{j \in\bbN} \subset \mathring W^{2m} (\Omega)$, $f, g \in L^2(\Omega)$, 
such that 
\begin{equation} 
\lim_{j \to \infty} \|f_j - f\|_{L^2(\Omega)} = 0 \, \text{ and } \,  
\lim_{j \to \infty} \big\|(-\Delta)^m f_j - g\big\|_{L^2(\Omega)} = 0.
\end{equation} 
Applying \eqref{3.19} to $u := (f_j - f_k) \in \mathring W^{2m} (\Omega)$, one infers for some 
$c_{m,\Omega} > 0$, 
\begin{equation}
\int_{\Omega} \big|\big((- \Delta)^m (f_j - f_k\big)(x)\big|^2 \, d^n x \geq c_{m,\Omega} 
 \|f_j - f_k\|_{m,\Omega}^2, \quad j, k \in \bbN, 
\end{equation}
implying that actually, $\{f_j\}_{j \in\bbN}$ is a Cauchy sequence in $\mathring W^{2m} (\Omega)$.  
By completeness of the latter space one concludes that $f \in \mathring W^{2m} (\Omega)$. Taking 
arbitrary $\psi \in C_0^{\infty}(\Omega)$,  one concludes that 
\begin{align}
& (g, \psi)_{L^2(\Omega)} = {}_{\cD(\Omega)'} \langle g, \psi \rangle_{\cD(\Omega)} 
= \lim_{j \to \infty} {}_{\cD(\Omega)'} \big\langle (- \Delta)^m f_j , \psi \big\rangle_{\cD(\Omega)} 
\no \\ 
& \quad = \lim_{j\to \infty} \int_{\Omega} \ol{f_j(x)} \, \big((- \Delta)^m \psi\big)(x) \, d^n x 
= \int_{\Omega} \ol{f(x)} \, \big((- \Delta)^m \psi\big)(x) \, d^n x    \no \\
& \quad =  {}_{\cD(\Omega)'} \big\langle (- \Delta)^m f , \psi \big\rangle_{\cD(\Omega)}.   
\end{align}
Hence, $g = (- \Delta)^m f$, implying closedness of $A_{\Omega,m}$. By the definition 
of $\mathring W ^{2m} (\Omega)$ (cf.\ \eqref{3.9}), $A_{\Omega,m}$ is the closure of 
$A_{min,\Omega,m}$. 

Strict positivity of $A_{min,\Omega,m}$, and hence that of $A_{\Omega,m}$, follows from 
\eqref{planch} and the Poincar\'e-type inequalities \eqref{3.10}. 
\end{proof}
  
 In the following we pick $m \in \bbN$ and denote by 
 $A_{K, \Omega, m}$ and $A_{F, \Omega, m}$ the Krein and 
 Friedrichs extension of $A_{\Omega,m}$ in $L^2(\Omega)$, respectively. By 
Corollary \ref{c3.8} (for $p=2$), $\dom(A_{\Omega,m}) = \mathring W^{2m} (\Omega)$ embeds 
 compactly into $L^2(\Omega)$ and hence by Lemma \ref{lB.2}, the operator 
 $A_{\Omega,m}^* A_{\Omega,m}$ has purely discrete spectrum, equivalently, the resolvent 
 of $A_{\Omega,m}^* A_{\Omega,m}$ is compact, in particular, 
 \begin{equation} 
\big(A_{\Omega,m}^* A_{\Omega,m}\big)^{-1} \in \cB_{\infty}\big(L^2(\Omega)\big), 
 \end{equation} 
 as the form associated with  $A_{\Omega,m}^* A_{\Omega,m}$ is given by 
 \begin{equation}
\mathfrak{a}_{\Omega,m}(f,g) :=(A_{\Omega, m}f,A_{\Omega, m}g)_{L^2(\Omega)}, \quad 
f,g\in\dom(\mathfrak{a}_{\Omega,m}):=\dom(A_{\Omega, m}).    \lb{3.30} 
 \end{equation}
Consequenty, also
 \begin{equation} 
|A_{\Omega,m}|^{-1} =  \big(A_{\Omega,m}^* A_{\Omega,m}\big)^{-1/2} \in 
\cB_{\infty}\big(L^2(\Omega)\big), 
 \end{equation}       
implying 
\begin{equation}
\big(\hatt A_{K, \Omega, m}\big)^{-1} \in \cB_{\infty}\big(L^2(\Omega)\big) 
\end{equation}  
 by \eqref{11.20a}. Thus,  
\begin{equation}
\sigma_{ess}(A_{K, \Omega, m}) \subseteq \{0\}.   
\end{equation} 
Let $\{\lambda_{K, \Omega, j}\}_{j\in\bbN}\subset(0,\infty)$ be the strictly positive eigenvalues 
of $A_{K,\Omega,m}$ enumerated in nondecreasing order, counting multiplicity, and let
\begin{equation}\label{4455}
N(\lambda,A_{K,\Omega,m}):=\#\{j\in\bbN\,|\,0<\lambda_{K,\Omega,j} < \lambda\}, \quad 
\lambda > 0, 
\end{equation}
be the eigenvalue distribution function for $A_{K,\Omega,m}$. Recalling the standard notation 
\begin{equation}
x_+ := \max \, (0, x), \quad x \in \bbR,
\end{equation}
then $N(\, \cdot \, ,A_{K,\Omega,m})$ permits the following estimate following the approach 
in \cite{La97}. 

\begin{theorem} \lb{t3.9}
Assume Hypothesis \ref{h3.3} and let $m \in \bbN$. Then one has the estimate,
\begin{equation}
N(\lambda ,A_{K,\Omega,m}) \leq (2 \pi)^{-n} v_n |\Omega| 
\{1 + [2m/(2m+n)]\}^{n/(2m)} \lambda^{n/(2m)}, \quad \lambda > 0,   
 \lb{3.31} 
\end{equation} 
where $v_n := \pi^{n/2}/\Gamma((n+2)/2)$ denotes the $($Euclidean$)$ volume of the unit 
ball in $\bbR^n$ $($$\Gamma(\cdot)$ being the Gamma function, cf.\ \cite[Sect.\ 6.1]{AS72}$)$. 
\end{theorem}
\begin{proof}
Following our abstract Section \ref{s2}, we introduce in addition to the symmetric form 
$\mathfrak{a}_{\Omega,m}$ in $L^2(\Omega)$ (cf.\ \eqref{3.30}), the form 
\begin{equation} 
\mathfrak{b}_{\Omega,m}(f,g) :=(f,A_{\Omega, m}g)_{L^2(\Omega)}, \quad 
 f,g\in\dom(\mathfrak{b}_{\Omega,m}):=\dom(A_{\Omega, m}).    
\end{equation}
By Lemma \ref{l2.7}, particularly, by \eqref{2.45}, one concludes that 
\begin{equation}
N(\lambda ,A_{K,\Omega,m}) \leq \max \big(\dim\, \big\{f \in \dom(A_{\Omega, m}) \,\big|\, 
 \mathfrak{a}_{\Omega,m}(f,f) - \lambda \, \mathfrak{b}_{\Omega,m}(f,f) < 0\big\}\big), \lb{3.32} 
\end{equation}
by also employing \eqref{2.47} and the fact that 
\begin{equation}
\mathfrak{a}_{\Omega,m}(f_{K,\Omega,j},f_{K,\Omega,j})) - \lambda \, 
\mathfrak{b}_{\Omega,m}(f_{K,\Omega,j},f_{K,\Omega,j}) 
= (\lambda_{K,\Omega,j} - \lambda) \|f_{K,\Omega,j}\|_{L^2(\Omega)}^2 < 0,
\end{equation}
where $f_{K,\Omega,j} \in \dom(A_{\Omega, m}) \backslash \{0\}$ aditionally satisfies 
\begin{align} 
\begin{split} 
& f_{K,\Omega,j} \in \dom(A_{\Omega, m}^* A_{\Omega, m}) \, \text{ and } \\ 
&A_{\Omega, m}^* A_{\Omega, m} f_{K,\Omega,j} 
= \lambda_{K,\Omega,j} \, A_{\Omega, m} f_{K,\Omega,j}. 
\end{split} 
\end{align}
To further analyze \eqref{3.32} we now fix $\lambda \in (0,\infty)$ and introduce the auxiliary operator 
\begin{align}
\begin{split}
& L_{\Omega,m, \lambda} := A_{\Omega, m}^* A_{\Omega, m} 
- \lambda \, A_{\Omega, m}, \\
& \dom(L_{\Omega,m,\lambda}) := \dom(A_{\Omega, m}^* A_{\Omega, m}). 
\end{split} 
\end{align}
By Lemma \ref{l3.1}, $L_{\Omega,m, \lambda}$ is self-adjoint, bounded from below, with purely 
discrete spectrum as its form domain 
\begin{equation} 
\dom\big(|L_{\Omega,m,\lambda} |^{1/2}\big) = \dom(A_{\Omega, m}) = \mathring W^{2m} (\Omega)
\end{equation}  
embeds compactly into $L^2(\Omega)$ by Corollary \ref{c3.8} (cf.\ Lemma \ref{lB.2}). 
We will study the auxiliary eigenvalue problem,
\begin{equation}
L_{\Omega,m, \lambda} \varphi_j = \mu_j \varphi_j, \quad 
\varphi_j \in \dom(L_{\Omega,m, \lambda}), 
\end{equation}
where $\{\varphi_j\}_{j \in \bbN}$ represents an orthonormal basis of eigenfunctions in 
$L^2(\Omega)$ and for simplicity of notation we repeat the eigenvalues $\mu_j$ of 
$L_{\Omega,m,\lambda}$ according to their multiplicity, assuming $\varphi_j$ to be linearly independent 
in the following. Since $\varphi_j \in \mathring W^{2m} (\Omega)$, we denote by 
\begin{equation}
\wti \varphi_j (x) := \begin{cases} \varphi_j(x), & x \in \Omega, \\ 0, & 
x \in \bbR^n \backslash \Omega,  \end{cases}
\end{equation} 
their zero-extension of $\varphi_j$ to all of $\bbR^n$ and note that
\begin{equation} 
\wti \varphi_j \in \mathring W^{2m} (\bbR^n), \quad 
\partial^{\alpha} \wti \varphi_j = \wti{\partial^{\alpha} \varphi_j}, \quad 0 \leq |\alpha| \leq 2m. 
\end{equation} 

Next, given $\mu > 0$, one estimates
\begin{equation}
\mu^{-1} \sum_{\substack{j \in \bbN \\ \mu_j < \mu}} (\mu - \mu_j) \geq 
\mu^{-1} \sum_{\substack{j \in \bbN, \\ \mu_j < 0, \, \mu_j < \mu}} (\mu - \mu_j) \geq 
\mu^{-1} \sum_{\substack{j \in \bbN, \\ \mu_j < 0, \, \mu_j < \mu}} \mu = n_-(L_{\Omega,m, \lambda}), 
\end{equation}
where $n_-(L_{\Omega,m, \lambda})$ denotes the number of strictly negative eigenvalues of 
 $L_{\Omega,m,\lambda}$. Combining, Lemma \ref{l3.2} and \eqref{3.32} one concludes that 
\begin{align}
& N(\lambda ,A_{K,\Omega,m}) 
\leq \max \big(\dim\, \big\{f \in \dom(A_{min, \Omega, m}) \,\big|\, 
 \mathfrak{a}_{\Omega,m}(f,f) - \lambda \, \mathfrak{b}_{\Omega,m}(f,f) < 0\big\}\big)  \no \\ 
& \quad = n_-(L_{\Omega,m, \lambda}) \leq 
\mu^{-1} \sum_{\substack{j \in \bbN \\ \mu_j < \mu}} (\mu - \mu_j) 
= \mu^{-1} \sum_{j \in \bbN} [\mu - \mu_j]_+, \quad \mu > 0.     \lb{3.44} 
\end{align} 
Next, we focus on estimating the right-hand side of \eqref{3.44}. 
\begin{align}
& N(\lambda ,A_{K,\Omega,m}) \leq \mu^{-1} \sum_{j \in \bbN} (\mu - \mu_j)_+ 
= \mu^{-1} \sum_{j \in \bbN} \big[(\varphi_j,(\mu - \mu_j) \varphi_j)_{L^2(\Omega)}\big]_+ \no \\
& \quad = \mu^{-1} \sum_{j \in \bbN} \Big[\mu \|\varphi_j\|_{L^2(\Omega)}^2  
- \big\|(- \Delta)^m \varphi_j\big\|_{L^2(\Omega)}^2 
+ \lambda \big(\varphi_j, (- \Delta)^m \varphi_j\big)_{L^2(\Omega)}\Big]_+   \no \\
& \quad = \mu^{-1} \sum_{j \in \bbN} \Big[\mu \|\wti \varphi_j\|_{L^2(\bbR^n)}^2  
- \big\|(- \Delta)^m \wti \varphi_j\big\|_{L^2(\bbR^n)}^2 + 
\lambda \big(\wti \varphi_j, (- \Delta)^m \wti \varphi_j\big)_{L^2(\bbR^n)}\Big]_+   \no \\
& \quad = \mu^{-1} \sum_{j \in \bbN} \bigg[\int_{\bbR^n} 
\big[\mu - \big(|\xi|^{4m} - \lambda |\xi|^{2m}\big)\big] 
\big|\hatt{\wti \varphi}_j(\xi)\big|^2 \, d^n \xi\bigg]_+   \no \\
& \quad \leq  \mu^{-1} \sum_{j \in \bbN} \int_{\bbR^n} 
\big[\mu - \big(|\xi|^{4m} - \lambda |\xi|^{2m}\big)\big]_+  
\big|\hatt{\wti \varphi}_j(\xi)\big|^2 \, d^n \xi    \no \\
& \quad =  \mu^{-1} \sum_{j \in \bbN} \int_{\bbR^n} 
\big[\mu - |\xi|^{4m} + \lambda |\xi|^{2m}\big]_+  
\big|\hatt{\wti \varphi}_j(\xi)\big|^2 \, d^n \xi    \no \\  
& \quad = \mu^{-1} \int_{\bbR^n} 
\big[\mu - |\xi|^{4m} + \lambda |\xi|^{2m}\big]_+  
 \sum_{j \in \bbN} \big|\hatt{\wti \varphi}_j(\xi)\big|^2 \, d^n \xi.    \lb{3.45} 
\end{align}
Here we used unitarity of the Fourier transform on $L^2(\bbR^n)$, the fact that 
$\big[\mu - |\xi|^{4m} + \lambda |\xi|^{2m}\big]_+$ has compact support (rendering 
the integral over a compact subset of $\bbR^n$), and the monotone convergence 
theorem in the final step.    

Next, one observes that 
\begin{align}
 \sum_{j \in \bbN} \big|\hatt{\wti \varphi}_j(\xi)\big|^2 &= (2 \pi)^{-n} \sum_{j \in \bbN} 
 \big|\big(e^{i \xi \cdot},  \wti \varphi_j\big)_{L^2(\bbR^n)}\big|^2 
 = (2 \pi)^{-n} \sum_{j \in \bbN} 
 \big|\big(e^{i \xi \cdot}, \varphi_j\big)_{L^2(\Omega)}\big|^2   \no \\
 &= (2 \pi)^{-n} \big\|e^{i \xi \cdot}\big\|_{L^2(\Omega)}^2 = (2 \pi)^{-n} |\Omega|,    \lb{3.46} 
\end{align}
employing the fact that $\{\varphi_j\}_{j \in \bbN}$ represents an orthonormal basis in 
$L^2(\Omega)$. 

Combining \eqref{3.45} and \eqref{3.46}, introducing $\alpha = \lambda^{-2}\mu$, changing 
variables, $\xi = \lambda^{1/(2m)} \eta$, and taking the minimum with respect to $\alpha > 0$, 
proves the bound, 
\begin{align}
& N(\lambda ,A_{K,\Omega,m}) \leq (2 \pi)^{-n} |\Omega| \min_{\alpha > 0}
\bigg(\alpha^{-1} 
\int_{\bbR^n} \big[\alpha - |\xi|^{4m} + |\xi|^{2m}\big]_+ d^n \xi\bigg) \lambda^{n/(2m)} ,    \no \\
& \hspace*{10cm} \lambda > 0.    \lb{3.81} 
\end{align} 
Explicitly computing the minimum over $\alpha > 0$ in \eqref{3.81}\footnote{Mark Ashbaugh 
generously provided us with the explicit value of the minimum in \eqref{3.81}.} finally yields the 
result \eqref{3.31}. 
\end{proof}

\section{Comparisons With Other Bounds and Weyl Asymptotics} 
\lb{s4}

In our final section we briefly discuss the bound \eqref{3.31} on the eigenvalue 
counting function $N(\lambda,A_{K,\Omega,m})$.

For smooth, bounded domains $\Omega \subset \bbR^n$, and smooth lower-order coefficients (not necessarily constant), Weyl asymptotics for $N(\lambda, A_{K,\Omega,m})$ 
as $\lambda \to \infty$ was first derived by Grubb \cite{Gr83}, 
\begin{equation}\lb{4.1}
N(\lambda,A_{K,\Omega,m})\underset{\lambda\to\infty}{=}
(2\pi)^{-n}v_n|\Omega|\,\lambda^{n/(2m)} + O\big(\lambda^{(n-\theta)/(2m)}\big), 
\end{equation} 
where $v_n := \pi^{n/2}/\Gamma((n+2)/2)$ denotes the (Euclidean) volume of the unit 
ball in $\bbR^n$ ($\Gamma(\dott)$ being the Gamma function, cf.\ \cite[Sect.\ 6.1]{AS72}), and 
\begin{equation}\lb{4.2}
\theta:=\max\,\Bigl\{\frac{1}{2}-\varepsilon\,,\,\frac{2m}{2m + n - 1}\Bigl\},
\,\text{ with $\varepsilon>0$ arbitrary}.
\end{equation} 
We also refer to \cite{Mi94}, \cite{Mi06}, and more recently, \cite{Gr12}, where the authors  derive a sharpening of the remainder in \eqref{4.1} to any $\theta<1$. In the case $m=1$, Weyl 
asymptotics for $N(\lambda, A_{K,\Omega,1})$ was derived in \cite{AGMT10} for (bounded) quasi-convex domains, and most recently, in \cite{BGMM14} for bounded Lipschitz domains.  

The power law behavior $\lambda^{n/(2m)}$ of the estimate \eqref{3.31} for general 
domains governed by Hypothesis \ref{h3.3} (no smoothness of $\Omega$ being asssumed at all 
in the case of bounded domains), coincides with that in the known Weyl asymptotics \eqref{4.1} 
and is of course consistent with the abstract estimate \eqref{2.24}. In this connection we note that  
Weyl-type asymptotics and estimates for $N(\lambda,A_{F,\Omega,m})$, and hence upper 
bounds for $N(\lambda,A_{K,\Omega,m})$, without regularity assumptions on $\Omega$ can be 
found, for instance, in \cite{BS70}, \cite{BS71}, \cite{BS72}, \cite{BS73}, \cite{BS79}, \cite{BS80}, 
\cite{Ge13}, \cite{GLW11}, \cite{HH08}, \cite{HH11}, \cite{La97}, \cite{Li80}, \cite{LY83}, \cite{Me77}, \cite{NS05}, \cite{Ro71}, \cite{Ro72}, \cite{Ro76}, \cite{Sa01}, \cite{We08}. We mention, in particular, 
the bound for $N(\lambda,A_{F,\Omega,m})$ derived in \cite{La97} (extending earlier results in 
\cite{LY83} in the case $m=1$) which reads 
\begin{equation}
N(\lambda,A_{F,\Omega,m}) \leq (2\pi)^{-n}v_n|\Omega| \, [1 + (2m/n)]^{n/(2m)} \lambda^{n/(2m)}, 
\quad \lambda > 0.    \lb{4.3} 
\end{equation}
A comparison of \eqref{4.3} with our result \eqref{3.31},  
\begin{equation}
N(\lambda,A_{K,\Omega,m}) \leq (2\pi)^{-n}v_n|\Omega| \, 
\{1 + [2m/(2m+n)]\}^{n/(2m)} \lambda^{n/(2m)}, 
\quad \lambda > 0,   \lb{4.3a} 
\end{equation}
clearly demonstrates the superiority of the buckling problem approach developed here over the bound obtained by combining the generally valid estimate \eqref{1.13} with \eqref{4.3} due to the extra term $2m$ in \eqref{4.3a} as compared to \eqref{4.3}. 

Additional comparisons between the bound \eqref{4.3a} and Weyl asymptotics, as well as an extension of our approach replacing $(-\Delta)^m$ by $(-\Delta + V)^m$, $m\in\bbN$, for an appropriate class of potentials $V \geq 0$ supported in $\ol \Omega$, will appear in 
\cite{AGLMS14}.

\medskip

\noindent 
{\bf Acknowledgments.} We are grateful to Mark Ashbaugh and Sergey Naboko for very 
helpful discussions. Especially, we are indebted to Mark Ashbaugh for communicating to us the explicit value of the minimum in \eqref{3.81}. We also thank the anonymous referee for a very 
careful reading of this manuscript, implying a number of improvements in Section \ref{s3}. In particular, this prompted us to present the compact embedding result, Theorem \ref{t3.7}. 

F.G.\ and A.L.\ are indebted to all organizers of QMath12, and particularly, 
to Pavel Exner, Wolfgang K\"onig, and Hagen Neidhardt, for fostering a very stimulating 
atmosphere during the conference, leading to this collaboration. 

 
\end{document}